\documentclass[12pt]{amsart}
\usepackage{amsmath,amsfonts,amsthm,amssymb,mathrsfs,amsxtra,amscd,latexsym}
\usepackage[hmargin=2cm,vmargin=3cm]{geometry}
\usepackage[usenames]{color}
\usepackage[all]{xy}
\usepackage{hyperref}
\usepackage[dvipsnames]{xcolor}
\usepackage[normalem]{ulem}

\definecolor{blue}{rgb}{0,0,1}
\definecolor{red}{rgb}{1,0,0}
\definecolor{green}{rgb}{0,.6,.2}
\definecolor{purple}{rgb}{1,0,1}

\long\def\red#1\endred{\textcolor{red}{#1}}
\long\def\blue#1\endblue{\textcolor{blue}{#1}}
\long\def\purple#1\endpurple{\textcolor{purple}{ #1}}
\long\def\green#1\endgreen{\textcolor{green}{#1}}

 \newtheorem{thm}{Theorem}[section]
 
 \newtheorem{cor}[thm]{Corollary}
 \newtheorem{lem}[thm]{Lemma}
 \newtheorem{prop}[thm]{Proposition}

 \newtheorem{rmk}[thm]{Remark}
 
 \theoremstyle{definition}
 \theoremstyle{remark}
 \numberwithin{equation}{section}

\newcommand{\mat}{\left(\begin{matrix}}
\newcommand{\emat}{\end{matrix}\right)}
\newcommand{\sm}{\left(\begin{smallmatrix}}
\newcommand{\esm}{\end{smallmatrix}\right)}

\def\CC{\mathbb{C}}
\def\HH{\mathbb{H}}

\def\QQ{\mathbb{Q}}
\def\RR{\mathbb{R}}
\def\ZZ{\mathbb{Z}}

\def\det{\mathrm{det}}

\def\im{\mathrm{Im}}

\def\SL{\mathrm{SL}}

\def\bsl{\backslash}

\begin{document}

\title{Values of harmonic weak Maass forms on Hecke orbits}

%%%%%%%%%%%%%%%%%%%%%%%%%%%%%%%%%%%%%%%%%%%%%%%%%%%%%%%%%%%%%%%%%%%%%%%%

\author{Dohoon Choi}
\author{Min Lee}
\author{Subong Lim}

\address{Department of Mathematics, Korea University, 145 Anam-ro, Seongbuk-gu, Seoul 02841, Republic of Korea}
\email{dohoonchoi@korea.ac.kr}	

\address{Howard House, University of Bristol, Queens Ave, BS8 1SN, United Kingdom}
\email{min.lee@bristol.ac.uk}

\address{Department of Mathematics Education, Sungkyunkwan University, Jongno-gu, Seoul 03063, Republic of Korea}
\email{subong@skku.edu}

\thanks{
D.~Choi\ was supported by Samsung Science and Technology Foundation under Project SSTF-BA1301-11.
M.~Lee\ was partially supported by
Royal Society University Research Fellowship ``Automorphic forms,
L-functions and trace formulas''. }
%S.~Lim\ was supported by NRF-2017R1C1B5017409.}



\begin{abstract}
Let $q:=e^{2 \pi iz}$, where $z \in \mathbb{H}$. For an even integer $k$, let $f(z):=q^h\prod_{m=1}^{\infty}(1-q^m)^{c(m)}$ be a meromorphic modular form of weight $k$ on $\Gamma_0(N)$.
For a positive integer $m$, let $T_m$ be the $m$th Hecke operator and $D$ be a divisor of a modular curve with level $N$.
Both subjects, the exponents $c(m)$ of a modular form and the distribution of the points in the support of $T_m. D$, have been widely investigated. 
%The exponents $c(m)$ of a modular form have been investigated in several works. For example, Borcherds proved that if $f$ has a Heegner divisor, then the $m$th exponent $c(m)$ is the $m^2$th coefficient of a fixed modular form of half integral weight.
%On the other hand, the points in the support of $T_m.D$ were investigated from several perspectives such as their distribution on the fundamental domain for $\Gamma_0(N)$ and the rank of a subgroup of $J_0(N)$ generated by Hecke points, and so on.

When the level $N$ is one, Bruinier, Kohnen, and Ono obtained, in terms of the values of $j$-invariant function, identities between the exponents $c(m)$ of a modular form  and the points in the support of $T_m.D$.
In this paper, we extend this result to general $\Gamma_0(N)$ in terms of values of harmonic weak Maass forms of weight $0$.  By the distribution of Hecke points, this applies to obtain an asymptotic behaviour of convolutions of sums of divisors of an integer and sums of exponents of a modular form.
\end{abstract}

\maketitle

%%%%%%%%%%%%%%%%%%%%%%%%%%%%%%%%%%%%%%%%%%%%%%%%%%%%%%%%%%%%%%%%%%%%%%%%
\section{Introduction} \label{section1}
%%%%%%%%%%%%%%%%%%%%%%%%%%%%%%%%%%%%%%%%%%%%%%%%%%%%%%%%%%%%%%%%%%%%%%%%
Let $\mathbb{H}$ be the complex upper half plane. For a positive integer $N$, let $Y_0(N)$ be the modular curve of level $N$ defined by $\Gamma_0(N) \backslash \mathbb{H}$, and $X_0(N)$ denote the compactification of $Y_0(N)$ by adjoining the cusps. Let $J_0(N)$ be the jacobian of a modular curve $X_0(N)$. We denote by $Div(C)$ the divisor group of a curve $C$.
If $f$ is a function on $C$ and $D=\sum_{P \in C} n_P P$ is a divisor of $C$, we define
\[
f(D):=\sum n_P f(P).
\]
The $m$th normalized Hecke operator $T_m$ acts on $Div(Y_0(N))$, and it is denoted by $T_n . D$ for $D\in Div(Y_0(N))$.
We call $T_m .D$ {\it the $m$th Hecke orbit of $D$}.
Especially, when $D$ is a divisor corresponding to $i \in \mathbb{H}$, a point in the support of $T_m.D$ is called {\it a Hecke point}. 
Hecke points have been investigated from several perspectives such as their distribution on the fundamental domain for $\Gamma_0(N)$ \cite{COU, CU, GM, GO}
and the rank of a subgroup of $J_0(N)$ generated by Hecke points \cite{S}, and so on.
 Let $q:=e^{2 \pi iz}$, where $z \in \mathbb{H}$. For an even integer $k$, let $f(z):=q^h\prod_{m=1}^{\infty}(1-q^m)^{c(m)}$ be a meromorphic modular form of weight $k$ on $\Gamma_0(N)$.
The exponents $c(m)$ of a modular form were investigated in various works (for examples, see \cite{Bor2, Bor, Z}).
For example, Borcherds \cite{Bor} proved that if $f$ has a Heegner divisor, then the $m$th exponent $c(m)$ is the $m^2$th coefficient of a fixed modular form of half integral weight.
Bruinier, Kohnen, and Ono \cite{BKO} obtained a connection between these exponents of a modular form and the points in the support of $T_n.D$.

For the modular invariant $j$, let $J(z):=j(z)-744$. For positive integers $k$ and $m$, let $\sigma_k(m):=\sum_{d|m} d^k$, and $\sigma_{f}(m):=\sum_{d|m} c(d)$.
Bruinier, Kohnen, and Ono \cite{BKO} proved the following identities between values $J(T_m . D_f)$ and sum of exponents  in the product expansion of $f$:
$$
\sum_{d|m}c(d)d=2k \sigma_1(m)+J(T_m.D_f)
$$
for every positive integer $m$, where $D_f$ denotes the divisor of $f$ on $X_0(N)$.
In other words, the value $J(T_m.D_f)$ can be expressed as the sum of the following values:
\begin{enumerate}
\item a multiple of the divisor function $\sigma_1(m)$,
\item the convolution of $\sigma_1(m)$ (sum of divisors) and $\sigma_f(m)$ (sum of exponents).
\end{enumerate}

%Let $D_f$ be a divisor of $Y_0(N)$ defined by the poles and zeros of $f$ with multiplicity. Let $g$ be a modular function on $\Gamma_0(N)$ or a harmonic  weak Maass form of weight $0$ on $\Gamma_0(N)$. Let
%\[
%\Phi_{f,g}(z):=\sum_{m=1}^\infty g(T_m . D_f)q^m
%\]
%be a formal power series in $\mathbb{C}[[q]]$.

%The above identities immediately imply that basically a formal power series {\color{blue} $\Phi_{f,J}$} is modular.
 They applied this result to prove the modularity of the generating series for $\sigma_f(m)$ and to obtain several $p$-adic properties of  $J(T_m . D_f)$ and exponents of a meromorphic modular form $f$. 
 Based on the argument  in \cite{BKO}, %of Bruinier, Kohnen and Ono,
the result was extended to several cases such as $\Gamma_0(N)$ with genus zero by Ahlgren \cite{A}, Jacobi forms by Choie and Kohnen \cite{CK}, and higher levels by the first author \cite{C2}. %, and so on.

For general positive integers $N$, the first author studied in \cite{C2} the generalization of \cite{BKO} to a harmonic  weak Maass form  $J_{N,1}$ of weight $0$ defined as a Poincar\'e  series (instead of a weakly holomorphic  modular form of weight $0$).
It was proved in \cite{C2} that the value $J_{N,1}(T_m . D_f)$ can be expressed as the sum of the following values:
\begin{enumerate}
\item a linear combination of the divisor functions $\sigma_1(nm)$ for $n|N$,
\item the convolution of $\sigma_1(m)$ (sum of divisors) and $\sigma_f(m)$ (sum of exponents),
\item the regularized Petersson inner product $R_{f,N}(m)$  of a meromorphic modular form and a cusp form.
\end{enumerate}

In this paper, we show that $R_{f,N}(m)$, the value of the regularized Petersson inner product in identities \cite{C2}, is zero, and so we give explicit identities between values $J_{N,1}(T_m . D_f)$ and sums of exponents in the product expansion of $f$. As an application, we obtain an asymptotic behavior for the convolution of $\sigma_1(m)$ (sum of divisors) and $\sigma_f(m)$ (sum of exponents) as $m \rightarrow \infty$.

Recently, Bringmann, Kane,  L\"obrich, Ono, and Rolen \cite{BKLOR} showed that for any fixed $N$ the generating series for $J_{N,1}(T_m . D_f)$ is basically modular.
Moreover, their result implies that 
there is a cusp form such that, for each $m$, $R_{f,N}(m)$, the value of regularized Petersson inner product, is given by the $m$th coefficient of a fixed cusp form.

%\begin{itemize}
%\item identities between values $J_{N,1}(T_m \cdot D_f)$ and sum of exponents of $f$ (for the definition of $J_{N,1}$, see (\ref{J_N})),
%\item asymptotic behaviour of $J_{N,1}(T_m\cdot D_f)$ (as $m \rightarrow \infty$),
%\item asymptotic behaviour of sums of exponents in the product expansion of $f$.
%\end{itemize}

Let $\mathcal{F}_1$ denote the usual  fundamental domain for the action of $\mathrm{SL}_2(\ZZ)$ on $\HH$ 
given 
%defined 
by
\[
\mathcal{F}_1 := \left\{z\in\HH\ \bigg|\ |z|>1,\  -\frac12 \leq \mathrm{Re}(z) <\frac12   \right\} \cup \biggl\{z\in\HH\ \bigg|\ |z|=1,\ \mathrm{Re}(z) \leq 0\biggr\}
\]
and
\[
\mathcal{F}_N := \bigcup_{\gamma\in \mathrm{SL}_2(\ZZ) \backslash \Gamma_0(N)} \gamma \mathcal{F}_1.
\]
Here we choose coset representatives for $\mathrm{SL}_2(\ZZ)\backslash\Gamma_0(N)$ such that
\[
\mathcal{F}_N \subset \left\{z\in \HH \ \bigg|\ |\mathrm{Re}(z)| \leq \frac{1}{2}\right\}.
\]
Then, $\mathcal{F}_N$ is a fundamental domain for the action of $\Gamma_0(N)$ on $\HH$.
%Let $\mathcal{F}_N$ be a fundamental domain for the action of $\Gamma_0(N)$ on $\mathbb{H}$
%such that
%\[
%\mathcal{F}_N = \bigcup_{\gamma\in \mathrm{SL}_2(\ZZ)\setminus \Gamma_0(N)} \gamma \mathcal{F}_1\ \text{and}\
%\mathcal{F}_N \subset \left\{ z\in \HH\ |\ |\mathrm{Re}(z)|\leq \frac12\right\}.
%\]
Let $\mathcal{C}_N$ be the set of inequivalent cusps of $\Gamma_0(N)$.
Let $k$ be an even integer and $f$ be a meromorphic modular form of weight $k$ on $\Gamma_0(N)$. For $\tau \in \mathbb{H} \cup \{ i \infty \} \cup \mathbb{Q}$, let $Q_{\tau}$ be the image of $\tau$ under the canonical map from $ \mathbb{H} \cup \{ i \infty \} \cup \mathbb{Q}$ to $X_0(N)$. For $\tau \in \mathbb{H} \cup \{ i \infty \} \cup \mathbb{Q}$, we denote by $\nu_{\tau}^{(N)}(f)$ the order of zero of $f$ at $Q_{\tau}$ on $X_0(N)$.
Let us note
\begin{equation*} \label{value_divisor}
g(T_m . D_f)=\sum_{\tau \in \mathcal{F}_N} \nu_{\tau}^{(N)}(f)g(T_m. \tau).
\end{equation*}
Moreover, for a divisor $D=\sum n_z Q_z$ of $X_0(N)$, we can give a more explicit expression of $T_m. D$.
For $\gamma=\sm a & b \\ c & d \esm  \in \mathrm{GL}_2(\mathbb{R})$
with positive determinant, we define the action of $\gamma$ for $z\in \mathbb{H}$ by
\[
\gamma z:=\frac{az+b}{cz+d}.
\]
For a positive integer $m$ prime to $N$, let
\[
T(m):=\{ \gamma=\sm a & b \\ 0 & d \esm  
%\in  \mathrm{M}_2(\ZZ)
%\mathrm{GL}_2(\mathbb{Z})
\; | \; a, b, d\in \ZZ, \; a>0, \; ad=m,\ \text{and}\ 0\leq b<d\}.
\]
Then, we have
\[
T_m. D=\sum n_z \sum_{\gamma \in T(m)} Q_{\gamma z}.
\]

Next, we define  the Ramanujan theta-operator by
\[
\theta(f)(z) :=\frac{1}{2\pi i} \frac{d}{dz} f(z).
\]
Let
\[
f_\theta(z):= \frac{\theta f(z)}{f(z)}-\frac{k}{12}E_2(z),
\]
where $E_2$ is the usual normalized Eisenstein series  of weight $2$ for $\mathrm{SL}_2(\ZZ)$.

Let $N>1$ and $I_v$ be the usual modified Bessel functions as in \cite{AS}.
For a positive integer $n$, we define the Poincar\'e series of weight $0$ and index $n$ by
\[
F_{N, n}(z,s) := \sum_{\gamma\in \Gamma_0(N)_{\infty}\backslash\Gamma_0(N)} \pi \left| n\mathrm{Im}(\gamma z)\right|^{1/2} I_{s-\frac12}(|2\pi n\mathrm{Im}(\gamma z)|)e(-n\mathrm{Re}(\gamma z)),
\]
where $s\in \CC$ with $\mathrm{Re}(s)>1$ and $e(z) := e^{2\pi iz}$.
Let  $j_{N,n}(z)$ be the continuation of $F_{N,n}(z,s)$ as $s\to 1$ from the right.
Then, the function $j_{N,n}$ is a harmonic weak Maass form of weight $0$ on $\Gamma_0(N)$ (see \cite[Section 2]{C2} for details).
Let
$
J_{N,n}(z):=j_{N,n}(z)-\beta_{N,n},
$
where $\beta_{N,n}$ is the constant term of the Fourier expansion of $j_{N,n}$ at the cusp $i\infty$.

For square-free $N$, 
let $D(N)$ be the number of divisors of $N$, and $\{d_1,d_2, \ldots, d_{D(N)-1}, N\}$ be the set of distinct divisors of $N$ such that $d_{i_1}<d_{i_2}$ if $i_1<i_2$. 
Let $A_N$ be the $(D(N)-1)\times (D(N)-1)$ matrix whose $ij$-entry $a_{ij}$ is defined by
$$
a_{ij}=\left(1-\frac{\gcd(d_i,d_j)^2}{d_j}\right).
$$
Let $A_{f,j}$ be a matrix obtained from $A_N$ by replacing the $j$th column of $A_N$ with a column matrix whose $i$th component is $\nu_{1/d_i}^{(N)}(f)-\frac{k}{12}$.
%Here, $\alpha_{1/d_i}$ denotes the width of $\Gamma_0(N)$ at the cusp $1/d_i$ (see Section 2 for the precise definition of $\alpha_{1/d_i}$). 
With this notation, we state our main theorem.

\begin{thm}\label{modularity}
Let $k$ be an even integer and $N>1$ be a positive integer. Suppose that
\begin{equation} \label{productexpansion}
f(z)= q^{h_\infty}\prod_{n=1}^\infty (1-q^n)^{c(n)}
\end{equation}
is a meromorphic modular form of weight $k$ on $\Gamma_0(N)$. Then
\[
-\sum_{\tau\in \mathcal{F}_N} \nu_\tau^{(N)}(f)-\sum_{m=1}^\infty \biggl( \sum_{\tau\in \mathcal{F}_N} \nu_\tau^{(N)}(f)J_{N,m}(\tau)\biggr) q^m = f_{\theta}(z)-\mathcal{E}_2(z),
\]
where $\mathcal{E}_2$ is a modular form in the Eisenstein space of weight $2$ on $\Gamma_0(N)$. Moreover, if $N$ is square free, then, for every positive integer $m$ prime to $N$,
\begin{equation} \label{J_exponents}
-J_{N,1}(T_m. D_f)=\sum_{d\mid m} dc(d) +24 \left( \sum_{1 \leq j \leq D(N)-1 }\frac{\det(A_{f,j})}{\det(A_N)}  + \frac{k}{12}\right) \sigma_1(m).
\end{equation}
\end{thm}

\begin{rmk}\label{Primelevel}
The modular form $\mathcal{E}_2$ in Theorem \ref{modularity} is determined by the order of zero or pole of $f$ at each cusp. In many cases, a modular form $\mathcal{E}_2$ can be expressed as a sum of explicit modular forms.
For example, if $N$ is square free, then
\[
\mathcal{E}_2(z)= \sum_{1 \leq j \leq D(N)-1 }\frac{\det(A_{f,j})}{\det(A_N)}(E_2(z)-d_jE_2(d_jz)).
\]
%\green\sout{where $h_\infty$ denotes the order of zero of $f$ at cusp $i\infty$.}\endgreen
%\red (moved to the statement of the theorem)\endred
\end{rmk}

Let $D:=\sum_{z \in S} n_z Q_z$ be a divisor of $Y_0(N)$, where $S$ is a finite set in $\mathcal{F}_N$. For a positive real number $r \geq 1$, we define a divisor $D_{>r}$ by
\[
D_{\; >r}=\sum_{z \in S \atop  \mathrm{Im}(\tilde{z})>r} n_z Q_{\tilde{z}}.
\]
Here, $\tilde{z}$ is a complex number in $\mathcal{F}_N$, which is equivalent to $z$ under the action of $\Gamma_0(N)$.
%Let $e(z):=e^{2 \pi iz}$. %, and $\mathcal{F}_{N}^{(r)}:=\{ z \in \mathcal{F}_N \; : \; \mathrm{Im}(z)\leq r\}$.
%Let $\theta$ be a constant towards the Ramanujan conjecture.
%The best-known bound is $\theta = \frac{7}{64}$ (Kim-Sarnak bound in \cite[Appendix 2]{KS}).
By the argument of Duke \cite{D} and equidistribution of Hecke points (\cite{GM}, \cite{COU} and \cite{CU}), Theorem \ref{modularity} implies the following theorem.

\begin{thm} \label{main2}
Let $k,N$, and  $f$ be given as in Theorem \ref{modularity}.
Assume that $N$ is square free.
Let $m$ be a positive integer prime to $N$, and  $h_f$ denote  the sum of the orders of zero or pole of $f$ at $Q_{\tau}$ on $Y_0(N)$.
%Suppose that $f(z) = q^{h_{\infty}} + \sum_{n=h_{\infty}+1}^\infty a(n)q^n$ is a meromorphic modular form of weight $k$ on $\Gamma_0(N)$.
%Let $h_f$ be the sum of orders of zero or pole of $f$ at $Q_{\tau}$ on $Y_0(N)$.
%Let $h_f$ be the sum of orders of zero or pole of $f$ at each cusp except $Q_{i \infty}$ on $X_0(N)$.
Then
\begin{eqnarray*}
&&\lim_{m \rightarrow \infty}
\frac{1}{\sigma_1(m)} \left(24 \left( \sum_{1 \leq j \leq D(N)-1 }\frac{\det(A_{f,j})}{\det(A_N)}  + \frac{k}{12}\right)\sigma_1(m)-\sum_{d\mid m}d c(d) - e\left(-\left({T_m.D_f}\right)_{\; >1}\right) \right)\\
&&\qquad  =\frac{3h_f}{\pi [\mathrm{SL}_2(\mathbb{Z}):\Gamma_0(N)]} \lim_{\epsilon\to0} \int_{\mathcal{F}_{N}(\epsilon)} J_{N,1}(z) \frac{dxdy}{y^2},% + O\left(m^{-\frac12+\theta+\epsilon'}\right),
\end{eqnarray*}
%for any $\epsilon'>0$,
%as $m\to\infty$,
where
$c(n)$ are complex numbers determined by (\ref{productexpansion}).
Here, $\mathcal{F}_N(\epsilon)$ is defined by $\mathcal{F}_N - \cup_{\tau\in \mathcal{C}_N} B_{\tau}(\epsilon)$, where $B_{\tau}(\epsilon)$ is given in (\ref{Btau}).
\end{thm}

Recently, Ali and Mani \cite{AM} proved an upper bound for exponents $c(m)$ in the product expansion of $f$.
The sum $\sum_{d\mid m}d c(d)$ looks like a kind of convolution of $\sigma_1(m)$ (a sum of divisors) and $\sigma_f(m)$ (a sum of exponents of $f$). The above inequality means that, as $m\rightarrow \infty$, this convolution has a similar asymptotic behavior as that of the sum of divisors of $m$ except its main term.

The remainder of the paper is organized as follows.
In Section~\ref{s:preliminaries}, we introduce some preliminaries for  meromorphic $1$-forms on $X_0(N)$.
In Section~\ref{s:petersson}, we provide some basic facts on regularized Petersson inner product, and prove that $f_\theta$ is orthogonal to every cusp form of weight $2$ on $\Gamma_0(N)$ with respect to regularized Petersson inner product if $f$ is a meromorphic modular form on $\Gamma_0(N)$.
In Section~\ref{s:heckepts}, we recall some results related to the distribution of Hecke points
%$L^2$-convergence
for $\Gamma_0(N)$.
In Section~\ref{s:proof}, we prove our main theorems: Theorems \ref{modularity} and \ref{main2}.

%%%%%%%%%%%%%%%%%%%%%%%%%%%%%%%%%%%%%%%%%%%%%%%%%%%%%%%%%%%%%%%%%%%%%%%%%%%%%%%%%%%%%%%%%%%%%%%%%%%%%%%%%%%%%%%%%%%%%%%%%%%%%%%%%%%%%%%%%%%%%%%%%%%%%%%%%%%%%

%%%%%%%%%%%%%%%                        section2

%%%%%%%%%%%%%%%%%%%%%%%%%%%%%%%%%%%%%%%%%%%%%%%%%%%%%%%%%%%%%%%%%%%%%%%%%%%%%%%%%%%%%%%%%%%%%%%%%%%%%%%%%%%%%%%%%%%%%%%%%%%%%%%%%%%%%%%%%%%%%%%%%%%%%%%%%%%%%

\section{Residues of a meromorphic $1$-form on $X_0(N)$}\label{s:preliminaries}
Let $f$ be a meromorphic modular form of weight $2$ on $\Gamma_0(N)$. Assume that $t$ is a cusp of $\Gamma_0(N)$. Let $\sigma_{t}\in \mathrm{SL}_2(\mathbb{Z})$ be a matrix such that $\sigma_{t} (i \infty)= t$, and $\Gamma_0(N)_{t}$ denote the stabilizer of the cusp $t$ in $\Gamma_0(N)$.
We define a positive integer $\alpha_t$ by
\[
\sigma_{t}^{-1} \Gamma_0(N)_{t} \sigma_{t}=\left \{\pm \left(
                                                     \begin{smallmatrix}
                                                       1 & \ell \alpha_t \\
                                                       0 & 1 \\
                                                     \end{smallmatrix}
                                                   \right)
     \; : \; \ell \in \mathbb{Z} \right\},
\]
%\red (Change: $k \to \ell$ since $k$ has been used for the weight) \endred
and we call $\alpha_t$ the width of $\Gamma_0(N)$ at the cusp $t$.
The Fourier expansion of $f$ at the cusp $t$ is given by
\[
(f|_2{\sigma_t)}(z)=\sum a_t(n)q^{n/\alpha_t},
\]
where $|_k$ denotes the usual weight $k$ slash operator.
If a cusp $t$ is equivalent to $i \infty$, the Fourier coefficients $a_t(n)$ of $f$ at the cusp $t$ are simply denoted by $a(n)$.

For $\tau \in \mathbb{H} \cup \{ i \infty \} \cup \mathbb{Q}$, let $Q_{\tau}$ be the image of $\tau$ under the canonical map from $ \mathbb{H} \cup \{ i \infty \} \cup \mathbb{Q}$ to $X_0(N)$.
Then, $fdz$ can be considered as a meromorphic 1-form on $X_0(N)$.
Thus, we denote by $\mathrm{Res}_{Q_{\tau}}fdz$ the residue of $f$ at $Q_{\tau}$ on $X_0(N)$.
Let $\mathrm{Res}_{\tau}f$ be the residue of $f$ at $\tau$ on $\mathbb{H}$. The description of $\mathrm{Res}_{Q_{ \tau}}fdz$ is given in terms of $\mathrm{Res}_{\tau}f$.
For $\tau \in \mathbb{H}$, let $e_{\tau}$ be the order of the isotropy subgroup of $\Gamma_0(N)$ at $\tau$.
Then, we have
\begin{equation}\label{f_residue}
	\mathrm{Res}_{Q_{\tau}} f\;dz
	= \begin{cases}
	\frac{1}{e_{\tau}} \mathrm{Res}_{\tau} f, & \text{ if } \tau\in \mathbb{H}, \\
	  \frac{1}{2\pi i}\alpha_{\tau} a_{\tau}(0), & \text{ if } \tau\in \mathcal{C}_N.
	\end{cases}
\end{equation}
%\begin{equation} \label{f_residue}
%\begin{aligned}
%  \mathrm{Res}_{Q_{\tau}}fdz&=\frac{1}{e_{\tau}}\mathrm{Res}_{\tau}f & \text{ if } \tau \in \mathbb{H}, \\
%  \mathrm{Res}_{Q_{\tau}}fdz&=\alpha_{\tau} a_{\tau}(0) & \text{ if } \tau \in \mathcal{C}_N.
%\end{aligned}
%\end{equation}

Let us note that if $k$ is an even integer and  $f$ is a meromorphic modular form of weight $k$ on $\Gamma_0(N)$, then $f_{\theta}$ is a meromorphic modular form of weight $2$ on $\Gamma_0(N)$.
The residue of $f_{\theta}$
at each point on $X_0(N)$ is determined by the order of its zero or pole of $f$ at that point. Let $\mathrm{ord}_{\tau}(f)$ be the order of the
zero or pole of $f$ at $\tau$ on $\mathbb{H}$.
Since we have
\[
(cz+d)^{-2}E_2\left(\frac{az+b}{cz+d}\right) = E_2(z) + \frac{12}{2\pi i}\cdot \frac{c}{cz+d}
\]
for all $\sm a&b\\c&d\esm\in \mathrm{SL}_2(\ZZ)$, we obtain
\[
 (f_{\theta}|_2{\sigma_t})(z)=\frac{\theta(f|_k {\sigma_t})(z)}{(f|_k {\sigma_t})(z)}-\frac{k}{12}E_2(z)
\]
for a cusp $t$.
Thus, we have
%\begin{equation} \label{f_theta_residue}
%	\mathrm{Res}_{Q_{\tau}} f_{\theta} \; dz = \frac{1}{2\pi i} \nu_{\tau}^{(N)}(f)
%	- \begin{cases}
%	0, & \text{ if } \tau\in \mathbb{H}, \\
%	\frac{k}{12}\alpha_{\tau}, & \text{ if } \tau\in \mathcal{C}_N.
%	\end{cases}
%\end{equation}
\begin{equation} \label{f_theta_residue}
\mathrm{Res}_{Q_{\tau}}f_{\theta}dz=
\begin{cases}
\frac{1}{2\pi i}\nu_{\tau}^{(N)}(f)%=\frac{1}{e_{\tau}}\mathrm{ord}_{\tau}(f)
& \text{ if } \tau \in \mathbb{H},\\
 \frac{\alpha_\tau}{2\pi i} \left(\nu_{\tau}^{(N)}(f)-\frac{k}{12}\right)  & \text{ if } \tau \in \mathcal{C}_N.
\end{cases}
\end{equation}

%%%%%
\section{Regularized Petersson inner product}\label{s:petersson}
Petersson defined an inner product of two cusp forms with the same weight.
The Petersson inner product was extended by Borcherds \cite{Bor3} to the case in which one of the two forms is a weakly holomorphic modular form.
In this section, following  \cite{Bor3} and \cite{C2}, we define regularized Petersson inner product of a cusp form and a meromorphic modular form with the same weight. We prove that if $f$ is a meromorphic modular form on $\Gamma_0(N)$, then the regularized Petersson inner product of  $f_\theta$ with any cusp form of weight $2$ on $\Gamma_0(N)$ is zero.

Let $k$ be an even integer and $f$ be a meromorphic modular form of weight $k$ on $\Gamma_0(N)$.
%Assume that $k \geq 2$.
Let $\mathrm{Sing}(f)$ be the set of singular points of $f$ on $\mathcal{F}_N$. For a positive real number $\varepsilon$,  an $\varepsilon$-disk $B_{\tau}(\varepsilon)$ at $\tau$ is defined by
\begin{equation} \label{Btau}
B_{\tau}(\varepsilon):=\left\{
\begin{array}{ll}
  \{ z \in \mathbb{H} \; : \; |z-\tau| < \varepsilon \}, & \text{if } \tau \in  \mathbb{H},  \\
  \{ z \in \mathcal{F}_N \; : \; \mathrm{Im}(\sigma_{\tau}z) > 1/\varepsilon \}, & \text{if } \tau \in  \{ i \infty\} \cup \mathbb{Q}.
\end{array}
\right.
\end{equation}
Let $\mathcal{F}_N(f, \varepsilon)$ be a punctured fundamental domain for $\Gamma_0(N)$ defined by
\begin{equation*}
\mathcal{F}_N(f, \varepsilon):=\mathcal{F}_N-\bigcup_{\tau \in \mathrm{Sing}(f) \cup \mathcal{C}_N} B_{\tau}(\varepsilon).
\end{equation*}

Let $g$ be a cusp form of weight $k$ on $\Gamma_0(N)$.
The regularized Petersson inner product $(f,g)_{reg}$ of  $f$ and $g$ is defined by
\[
(f,g)_{reg}:=\lim_{\varepsilon \rightarrow 0} \int_{\mathcal{F}_N(f,\varepsilon)}f(z)\overline{g(z)}
\frac{dxdy}{y^{k-2}}.
\]
Then, we have the following proposition.

\begin{prop}\label{regpetersson}
Let $k$ be an even integer, and  $f$ be a meromorphic modular form of weight $k$ on $\Gamma_0(N)$. Then, for every cusp form $g$ of weight $2$ on $\Gamma_0(N)$,
\[
(f_{\theta},g)_{reg}=0.
\]
\end{prop}

\begin{proof}
Let $\Delta(z):=q\prod_{n=1}^{\infty}(1-q^n)^{24}$ be the unique normalized cusp form of weight 12 on $\mathrm{SL}_2(\mathbb{Z})$. Let
\[
F(z):=\frac{f(z)^{12}}{\Delta(z)^{k}}.
\]
Then, we have
\[
d((\log_{e}|F(z)|^2) \overline{g(z)}d\overline{z})=\frac{\partial_z F(z) \overline{F(z)} }{F(z)  \overline{F(z)}} \overline{g(z)} dz d\overline{z}=\frac{\partial_z F(z)}{F(z)} \overline{g(z)} (-2i)dxdy.
\]
Let us note that $\Delta$ has no zeros and no poles on $\mathbb{H}$.
Therefore, according to \cite[Theorem~1]{BKO}, we have
\[
\frac{\theta (\Delta)}{\Delta} = E_2.
\]
The function $\partial_z F(z)/F(z)$ is given as
\[
\frac{\partial_z F(z)}{F(z)}=12 \frac{\partial_z f(z)}{f(z)}- k \frac{\partial_z \Delta(z)}{\Delta(z)}= 12 \frac{\partial_z f(z)}{f(z)}- k (2 \pi i) E_2(z)=(24 \pi i) f_{\theta}(z).
\]
Thus, we have
\begin{equation}\label{differential}
d((\log_{e}|F(z)|^2) \overline{g(z)}d\overline{z})=(48 \pi)f_{\theta}(z)\overline{g(z)} dxdy.
\end{equation}

%Then, $\partial \mathcal{F}_N$ represents the boundary of $\mathcal{F}_N$ in $\mathbb{C}$.
In order to apply the Stokes theorem, we describe the boundary of $\mathcal{F}_N$. For a positive real number $\varepsilon$,  we define
\[
\gamma_{\tau}(\varepsilon):=\left\{
\begin{array}{ll}
  \{ z \in \mathbb{H} \; : \; |z-\tau| = \varepsilon \} & \text{if } \tau \in  \mathbb{H},  \\
  \{ z \in \mathcal{F}_N \; : \; \mathrm{Im}(\sigma_{\tau}z) = 1/\varepsilon \} & \text{if } \tau \in  \{ i \infty\} \cup \mathbb{Q}.
\end{array}
\right.
\]
Assume that $\varepsilon$ is sufficiently small. If $\partial^* \mathcal{F}_N(f, \varepsilon)$ denotes the closure of the set $\partial \mathcal{F}_N(f, \varepsilon)-\partial \mathcal{F}_N$ in $\mathbb{C}$, then
\begin{equation}\label{boundary}
\partial^* \mathcal{F}_N(f, \varepsilon)=\bigcup_{\tau \in \mathrm{Sing}(f) \cup \mathcal{C}_N} \gamma_{\tau}(\varepsilon),
\end{equation}
where $\partial D$ denotes the boundary of $D$ for a subset $D$ of $\mathbb{C}$.
From (\ref{differential}) and (\ref{boundary}), the Stokes theorem implies
\begin{eqnarray*}
\int_{\mathcal{F}_N(f_\theta,\varepsilon)}f_{\theta}(z)\overline{g(z)}dxdy&=& \int_{\partial^* \mathcal{F}_N(f_\theta, \varepsilon)} \frac{1}{48 \pi} (\log_{e}|F(z)|^2) \overline{g(z)}d\overline{z}\\
&=&\sum_{\tau \in \mathrm{Sing}(f_\theta) \cup \mathcal{C}_N} \int_{\gamma_{\tau}(\varepsilon)} \frac{1}{48 \pi} (\log_{e}|F(z)|^2) \overline{g(z)}d\overline{z}.
\end{eqnarray*}

For each $\gamma \in \mathrm{SL}_2(\mathbb{Z})$, the absolute value $|(g|_2{\gamma})(z)|$ exponentially decays as $\mathrm{Im}(z) \rightarrow \infty$, since $g$ is a cusp form. Thus, if $\tau \in \mathcal{C}_N$, then $\lim_{\varepsilon \rightarrow 0}\int_{\gamma_{\tau}(\varepsilon)} \frac{1}{48 \pi} (\log_{e}|F(z)|^2) \overline{g(z)}d\overline{z}=0$.

To complete the proof, we assume that $\tau \in \mathrm{Sing}(f_\theta)$. Then
\begin{align*}
&\left | \int_{\gamma_{\tau}(\varepsilon)} \frac{1}{48 \pi} (\log_{e}|F(z)|^2) \overline{g(z)}d\overline{z} \right| \\
&\leq \int_{\gamma_{\tau}(\varepsilon)} \frac{1}{48 \pi} |(\log_{e}|F(z)|^2)| |\overline{g(z)}| |d\overline{z}| \\
&\leq \max\{ |(\log_{e}|F(z)|^2)| \; : \; z \in \gamma_{\tau}(\varepsilon) \} M_1  \int_{\gamma_{\tau}(\varepsilon)}|d\overline{z}| \;\; (\text{some constant } M_1) \\
&\leq  \max\{ |(\log_{e}|F(z)|^2)| \; : \; z \in \gamma_{\tau}(\varepsilon) \} M_1 (2\pi \varepsilon).
\end{align*}
The function $F(z)$ can be expressed around $\tau$ as
\[
 F(z)=(z-\tau)^{12 \nu_\tau^{(N)}(f)}F_0(z),
\]
where $F_0(z)$ is a nowhere vanishing holomorphic function around $\tau$. If $\varepsilon$ is sufficiently small, then, for any $z \in \gamma_{\tau}(\varepsilon)$  we have
\begin{align*}
|(\log_{e}|F(z)|^2)|&\leq |(\log_{e}|(z-z_0)|^{24
\nu_\tau^{(N)}(f)})|+|(\log_{e}|F_0(z)|^2)| \\
&\leq |(\log_{e}|(z-z_0)|^{24\nu_\tau^{(N)}(f)})|+ M_2 \;\; (\text{some fixed constant } M_2)\\
&= |24\nu_\tau^{(N)}(f)\log_{e} \varepsilon|+M_2.
\end{align*}
Thus, for sufficiently small $\varepsilon$, we have
\[
\left | \int_{\gamma_{\tau}(\varepsilon)} \frac{1}{48 \pi} (\log_{e}|F(z)|^2) \overline{g(z)}d\overline{z} \right| \leq (|24\nu_\tau^{(N)}(f)\log_{e} \varepsilon|+M_2)  M_1 (2\pi \varepsilon).
\]
This implies that, for $\tau \in \mathrm{Sing}(f_\theta)$,
\[
\lim_{\varepsilon \rightarrow 0}  \int_{\gamma_{\tau}(\varepsilon)} \frac{1}{48 \pi} (\log_{e}|F(z)|^2) \overline{g(z)}d\overline{z} =0.
\]
Thus, we complete the proof.
\end{proof}

\section{Equidistribution of Hecke points}\label{s:heckepts}
%\section{$L^2$-convergence for $\Gamma_0(N)$}
%Let $\Gamma=\Gamma_0(N)$.
Let $\{u_j\}_{j\geq 0}$ be an orthonormal basis of the residual and cuspidal spaces of $L^2(\Gamma_0(N) \backslash\HH)$,
i.e., $u_0$ is a constant with the eigenvalue $\lambda_0 = 0$ and
$u_j$ is a Maass form  for $\Gamma_0(N)$
with eigenvalue $\lambda_j=s_j(1-s_j)$ for $j\geq 1$.
Further, assume that $\lambda_j$ are ordered so that $0< \lambda_1 \leq \lambda_2 \leq \cdots$.
For each cusp $t\in \QQ\cup\{\infty\}$, let $E_{t}(z, s)$ be the Eisenstein series at $t$ for $\mathrm{Re}(s)>1$, 
which is given by
\begin{equation*}
E_{t}(z,s)=\sum_{\gamma\in \Gamma_0(N)_{t} \bsl \Gamma_0(N)} (\mathrm{Im}(\sigma_{t}^{-1}\gamma z))^s.
\end{equation*}
Here, $\Gamma_0(N)_{t}\subset \Gamma_0(N)$ is the stability group of $t$.
For the properties of $E_{t}(z, s)$, see \cite[\S15]{IK}.

According to \cite[Theorem~15.5]{IK}, any $f\in L^2(\Gamma_0(N)\backslash \mathbb{H})$ has the spectral decomposition
\begin{equation*}
f(z) = \sum_{j\geq 0} \left<f, u_j\right> u_j(z)
+ \sum_{t\in \mathcal{C}_N} \frac{1}{4\pi} \int_{\mathbb{R}} \left<f, E_{t}(*, 1/2+ir)\right> E_{t}(z, 1/2+ir) \; dr
\end{equation*}
(valid in $L^2$-sense) and converges absolutely and uniformly on compact sets if $f$ and $\Delta f$ are smooth and bounded.

We now follow the proof of \cite[Theorem~3.1]{GM}.
Let
\begin{align*}
& f_C = \left<f, u_0\right> = \text{ the projection of $f$ onto the constant subspace}, \\
& f_M(z) = \sum_{j\geq 1} \left< f, u_j\right> u_j(z), \\
& f_E(z) = \sum_{t\in \mathcal{C}_N} \frac{1}{4\pi} \int_{\RR} \left< f, E_{t}(*, 1/2+ir)\right> E_{t}(z, 1/2+ir)\; dr. 
\end{align*}
Note that
\begin{equation} \label{constant_part}
f_C = \left<f, u_0\right> = \int_{\mathcal{F}_N} f(z) \; d\mu(z),
\end{equation}
where $d\mu(z) := \frac{3}{\pi [\mathrm{SL}_2(\mathbb{Z}):\Gamma_0(N)]}  \cdot \frac{dxdy}{y^2}$
 is the normalized Haar measure; so, $\int_{\mathcal{F}_N} \; d\mu(z)=1$.

Let
%\[
%\sigma_1(n) n^{-\frac{1}{2}} T_n u_j = \lambda_j(n)
%\]
$\lambda_j(n)$
be the $n$th Fourier coefficient of $u_j$.
By the Ramanujan conjecture, there exists $\theta\geq0$ such that
$|\lambda_j(n)| \leq c n^{\theta+\epsilon}$,
for any $\epsilon>0$.
So, we get
\begin{equation} \label{Eisenstein_part}
\frac{1}{\sigma_1(n)}\|T_n f_M\|_2 \leq c n^{-\frac{1}{2}+\theta+\epsilon} \|f_M\|_2.
\end{equation}
Note that the value of $\theta$ has been lowered to $\frac{7}{64}$ by Kim and Sarnak \cite[Appendix 2]{KS}.

%%%%%
In \cite[\S6, \S7 and \S8]{Y}, an explicit change-of-basis formula between the Eisenstein series attached to cusps 
and newform Eisenstein series attached to pairs of primitive Dirichlet characters is described.
The Eisenstein series attached to a Dirichlet character is an eigenfunction of Hecke operators $T_n$ for $\gcd(n, N)=1$,
and the absolute values of the corresponding eigenvalues are bounded above by $\sigma_0(n)n^{-\frac{1}{2}}$.
So, we get
\begin{equation} \label{Maass_part}
\frac{1}{\sigma_1(n)}\|T_n f_E\|_2\leq c n^{-\frac{1}{2}+\epsilon} \|f_E\|_2.
\end{equation}

If we combine (\ref{constant_part}), (\ref{Eisenstein_part}), and (\ref{Maass_part}), then we obtain the following theorem.
For more general result, see \cite{COU}.

%%%%%
\begin{thm} \label{L2_convergence}
Let $f\in L^2(\Gamma_0(N)\backslash \mathbb{H})$.
For a positive integer $n$ prime to $N$, we have
\[
\left\| \frac{1}{\sigma_1(n)}T_n f - \int_{\mathcal{F}_N} f(z) \; d\mu(z) \right\|_2
\leq c_\epsilon  n^{-\frac{1}{2}+\theta+\epsilon} \|f\|_2
\]
for any $\epsilon>0$.
The constant $c_{\epsilon}$ depends on $\epsilon$.
%Here, $\theta$ is a constant towards the Ramanujan conjecture.
%The best-known bound is $\theta= \frac{7}{64}$ (Kim-Sarnak bound in \cite[Appendix 2]{KS}).
\end{thm}
%%%%%

The pointwise convergence can be derived from \cite[Proposition~8.2]{CU}.
Note that elliptic differential operators are differential operators that generalize the Laplace-Beltrami operator $\Delta$.
For an integer $m\geq 2$, assume that $f, \Delta^m f\in L^2(\Gamma_0(N) \backslash \mathbb{H})$.
Then, by \cite[Proposition~8.2]{CU}, for a compact subset $\omega\subset \mathcal{F}_N$, there exist constants $C_1(\omega)$ and $C_2(\omega)$ such that, for any $z_0\in \omega$
\begin{multline*}
\left| \frac{1}{\sigma_1(n)}T_n f(z_0)- \int_{\mathcal{F}_N} f(z) \; d\mu(z) \right|
\\
\leq C_1(\omega) \left\| \frac{1}{\sigma_1(n)}T_n f-\int_{\mathcal{F}_N} f(z) \; d\mu(z) \right\|_2
+ C_2(\omega) \left\| \frac{1}{\sigma_1(n)}T_n (\Delta^m f)-\int_{\mathcal{F}_N} (\Delta^m f)(z) \; d\mu(z) \right\|_2.
\end{multline*}
So, we have the following corollary.

\begin{cor} \label{pointwise_convergence}
%Let $\theta$ be given as in Theorem \ref{L2_convergence}.
Assume that $f, \Delta^2 f\in L^2(\Gamma_0(N) \backslash \mathbb{H})$.
Take a compact $\omega\subset \Gamma_0(N)\backslash \mathbb{H}$ and a positive number $\epsilon$.
Then, there exists a constant  $C_{\omega, \epsilon}$ depending on  $\omega$ and $\epsilon$,
such that, for a positive integer $n$ prime to $N$, for any $z_0\in \omega$,
\[
\left| \frac{1}{\sigma_1(n)}T_n f(z_0)- \int_{\mathcal{F}_N} f(z) \; d\mu(z) \right|
\leq  C_{\omega, \epsilon}  n^{-\frac{1}{2}+\theta+\epsilon}  \max\{\|f\|_2, \|\Delta^2f\|_2\}.
\]
%where $\theta$ is the Kim-Sarnak bound.
\end{cor}

%%%%%
\section{Proofs}\label{s:proof}
Let $M^{Eis}_k(\Gamma_0(N))$ be the space of modular forms orthogonal to all the cusp forms of weight $k$ on $\Gamma_0(N)$,
%. This space $M^{Eis}_k(\Gamma_0(N))$
which  is called the Eisenstein space of weight $k$ on $\Gamma_0(N)$. In the following lemma, we prove that if $N$ is square-free, then, for a positive integer  $n$  prime to $N$, the $n$th coefficient of a modular form in $M^{Eis}_2(\Gamma_0(N))$ is a multiple of  $\sigma_1(n)$.
Recall the notations $D(N)$, $d_j$, and $A_N$ from Section \ref{section1}.
%Let us recall that $D(N)$ denotes the number of divisors of $N$, and that $\{1, d_1, \cdots, d_{D(N)-1}\}$ denotes the set of distinct divisors of $N$ such that $d_j< d_i$ if $j<i$. Further, recall that $A_N$ denotes the matrix consisting of its $ij$-entry $a_{ij}$ defined by
%$$
%a_{ij}=\frac{N}{d_i}\left(1-\frac{\gcd(d_i,d_j)^2}{d_j}\right).
%$$
Now, we prove the following lemma related to properties for modular forms in an Eisenstein space.

\begin{lem} \label{basis}
Suppose that $\mathcal{E}_2(z):=\sum_{n=0}^{\infty} b(n)q^n$ is a modular form in $M^{Eis}_2(\Gamma_0(N))$, and that $N$ is square free.
Then, the following statements are true.
\begin{enumerate}
\item There exists a constant $c$ such that for every positive integer $n$ prime to $N$,
\[
b(n)=c\sigma_1(n).
\]

\item Assume that the constant term of $\mathcal{E}_2(z)$ at cusp $1/d_i$ is $c_{d_i}$. Let $A_{j}$ be the matrix obtained from $A$ by replacing the $j$th column of $A$ with a column matrix whose $i$th component is $c_{d_i}$. Then
    \[
    \mathcal{E}_2(z)= \sum_{1 \leq j \leq D(N)-1 }\frac{\det(A_{j})}{\det(A_N)}(E_2(z)-d_jE_2(d_jz)).
    \]
\end{enumerate}
\end{lem}

\begin{proof}
(1) We claim that there is a basis of $M^{Eis}_2(\Gamma_0(N))$ consisting of modular forms $E_2(z)-d E_2(dz)$, where $d\neq 1$ are the divisors of $N$. Assume that the claim is true.
Then, $\mathcal{E}_2(z)$ can be expressed as a linear combination of $E_2(z)-d_jE_2(d_jz)$ having the form
\[
\mathcal{E}_2(z)=\sum_{1 \leq j \leq D(N)-1} a_j (E_2(z)-d_jE_2(d_jz)).
\]
Recall that $E_2$ has the Fourier expansion of the form
\begin{equation}\label{e:E2_fourier}
E_2(z) = 1-24\sum_{n=1}^\infty \sigma_1(n) q^n.
\end{equation}
Then,  the $n$th coefficient of $\mathcal{E}_2(z)$ is given by
\[
-24\sum_{1 \leq j \leq D(N)-1} a_j (\sigma_1(n) - d_j\sigma_1(n/d_j))
\]
for $n>0$,
and $a_j$ does not depend on $n$. Here, $\sigma_1(n/d) = 0$ if $n$ is not divisible by $d$.
%Let us note that if $\gcd(n,N)=1$ and $d \neq 1$, then the $n$th coefficient of $E_2(dz)$ is zero.
Thus, we have the proof of the lemma.

Now, we prove the claim. Suppose that
\[
\sum_{1 \leq j \leq D(N)-1} a_j (E_2(z)-d_jE_2(d_jz))=0,
\]
where $a_j$ are complex numbers.
We assume that complex numbers $a_j$ are not all zero.
Then, we have
\[
\sum_{1 \leq j \leq D(N)-1} a_j E_2(z)
=\sum_{1 \leq j \leq D(N)-1} a_j d_jE_2(d_jz).
\]
Comparing the $n$th coefficients of the forms on both sides for $n$ prime to $N$, we have
\[
\sum_{1 \leq j \leq D(N)-1}a_j E_2(z)=\sum_{1 \leq j \leq D(N)-1} a_j d_j E_2(d_jz)=0.
\]
Take the smallest positive integer $d_{j_0}|N$ such that $a_{j_0} \neq 0$.
Then, we have
\[
-a_{j_0}d_{j_0} E_2(d_{j_0} z)=\sum_{1 \leq j \leq D(N)-1} a_j d_jE_2(d_jz)-a_{j_0}d_{j_0} E_2(d_{j_0} z).
\]
Comparing the $d_{j_0}$th coefficients of the forms on both sides, we have $a_{j_0}=0$.
This is a contradiction.
%Repeating this process, we have
%all of $\alpha_d$ are zero.
Therefore, the modular forms $E_2(z)-dE_2(dz)$, $d|N$ and $d\neq1$, are linearly independent.

Let us note
\[
\dim_{\mathbb{C}}M^{Eis}_2(\Gamma_0(N))=D(N)-1.
\]
since $N$ is square free.
Thus, $$\{ (E_2(z)-dE_2(dz) \; : \; d\mid N\text{ and } d \neq 1\}$$ is a basis of $M^{Eis}_2(\Gamma_0(N))$. This completes the proof of the claim.

(2) From the proof of (1), we may assume that
\[
\mathcal{E}_2(z)=\sum_{1 \leq j \leq D(N)-1} a_j (E_2(z)-d_j E_2(d_j z)).
\]
Let us note that
$E_2(z)-\frac{3}{\pi \mathrm{Im}(z)}$ is a non-holomorphic modular form of weight $2$ on $\mathrm{SL}_2(\mathbb{Z})$.
By direct computation, there are $\gamma \in \mathrm{SL}_2(\mathbb{Z})$ and $\mu_j \in \mathbb{Z}$ such that
\[
\left(
    \begin{array}{cc}
      d_j & 0 \\
      0 & 1 \\
    \end{array}
  \right)
  \left(
    \begin{array}{cc}
      1 & 0 \\
      d_i & 1 \\
    \end{array}
  \right)
= \gamma
\left(
    \begin{array}{cc}
      1 & \mu_j \\
      0 &  d_j/\gcd(d_j,d_i) \\
    \end{array}
  \right)
  \left(
    \begin{array}{cc}
      \gcd(d_j,d_i) & 0 \\
      0 & 1 \\
    \end{array}
  \right)
\]
Thus,
\[
(E_2(z)-d_jE_2(d_jz))|_2 \sm 1 & 0 \\ d_i & 1 \esm = E_2(z)-\frac{\gcd(d_j,d_i)^2}{d_j}E_2\left(\frac{\gcd(d_j,d_i)^2}{d_j}z+\frac{\mu_j\gcd(d_j,d_i)}{d_j} \right).
\]
This implies that $a_j$ are the solution of the system
\[
c_{d_i}=\sum_{ 1 \leq j \leq D(N)-1} \left(1-\frac{\gcd(d_j,d_i)^2}{d_j}\right)a_j
\]
for $ 1 \leq i \leq D(N)-1$. Thus, the Cramer's rule completes the proof.
\end{proof}

Now, we prove Theorem \ref{modularity}.

\begin{proof}
Note that
\[
\left\{E_2(z) - dE_2(dz)\ |\ d\mid N,\ d\neq 1\right\}
\]
forms a basis of $ M_2^{Eis}(\Gamma_0(N))$ by the proof of Lemma \ref{basis}.
Therefore, we can take a modular form $\mathcal{E}_2  \in M_2^{Eis}(\Gamma_0(N))$
%\sout{of weight 2 on $\Gamma_0(N)$ such that $\mathcal{E}_2$ is in the Eisenstein space $M_2^{Eis}(\Gamma_0(N))$ and }
 such that  the constant term of $\mathcal{E}_2$ at each cusp except cusps equivalent to $i \infty$ is the same as that of $f_{\theta}$.
Suppose that  $\mathcal{E}_2$ has the Fourier expansion of the form
\[
\mathcal{E}_2(z)=\sum_{n=0}^{\infty} b(n)q^n.
\]

Note that, by (\ref{f_residue}) and (\ref{f_theta_residue}), we have
\[
\frac{2\pi i}{e_\tau} \mathrm{Res}_\tau f_\theta = \nu_\tau^{(N)}(f)
\]
for $\tau\in \mathcal{F}_N$.
Thus, from \cite[Lemma~3.1]{C2}, we obtain
\begin{equation} \label{small_j}
(f_{\theta}-\mathcal{E}_2, \xi_0(j_{N,m}))_{reg} = \beta_{N,m}(a_{\theta}(0)-b(0))+
a_\theta(m)-b(m)+\sum_{\tau\in \mathcal{F}_N} \nu_\tau^{(N)}(f)j_{N,m}(\tau),
\end{equation}
where $a_\theta(m)$ is the $m$th Fourier coefficient of $f_{\theta}$ and
$\xi_0$ is a differential operator defined by
\[
\xi_0(f)(z) := 2i\overline{\frac{\partial}{\partial\bar{z}} f(z)}.
\]
By the same argument in the proof of \cite[Lemma~3.1]{C2}, we have
\begin{equation} \label{large_j}
(f_{\theta}-\mathcal{E}_2, \xi_0(J_{N,m}))_{reg} = a_\theta(m)-b(m)+\sum_{\tau\in \mathcal{F}_N} \nu_\tau^{(N)}(f)J_{N,m}(\tau).
\end{equation}
Note that $\xi_0(j_{N,m}) = \xi_0(J_{N,m})$ since $J_{N,m}(z) = j_{N,m}(z)-\beta_{N,m}$.
Therefore, from (\ref{small_j}) and (\ref{large_j}), we have
\begin{equation} \label{constant}
a_\theta(0)-b(0)=\frac{1}{\beta_{N,m}} \left(\sum_{\tau\in \mathcal{F}_N} \nu_\tau^{(N)}(f)J_{N,m}(\tau)-\sum_{\tau\in \mathcal{F}_N} \nu_\tau^{(N)}(f)j_{N,m}(\tau)\right)=-\sum_{\tau\in \mathcal{F}_N} \nu_\tau^{(N)}(f).
\end{equation}

Proposition \ref{regpetersson} implies that
\[
(f_{\theta}-\mathcal{E}_2, \xi_0(j_{N,m}))_{reg}=0.
\]
Therefore, from (\ref{large_j}), we have
\begin{equation} \label{nonconstant}
a_\theta(m)-b(m)=-\sum_{\tau\in \mathcal{F}_N} \nu_\tau^{(N)}(f)J_{N,m}(\tau)
\end{equation}
for every positive integer $m$.
Thus, from (\ref{constant}) and (\ref{nonconstant}), we obtain
\begin{equation} \label{result}
 f_{\theta}(z)-\mathcal{E}_2(z) = -\sum_{\tau\in \mathcal{F}_N} \nu_\tau^{(N)}(f)-\sum_{m=1}^\infty \biggl( \sum_{\tau\in \mathcal{F}_N} \nu_\tau^{(N)}(f)J_{N,m}(\tau)\biggr) q^m.
\end{equation}

%Let us note that $E_2$ has the Fourier expansion of the form
%\[
%E_2(z)=1-24\sum_{n=1}^{\infty}\sigma_1(n)q^n.
%\]
 By \eqref{productexpansion} and the Fourier expansion of $E_2$ given in \eqref{e:E2_fourier},
%Thus, by (\ref{productexpansion}),
$f_{\theta}$ has the Fourier expansion of the form
\begin{equation} \label{f_theta_expansion}
f_{\theta}(z)=h_\infty+\sum_{n=1}^{\infty}\sum_{d\mid n}d c(d)q^n-\frac{k}{12}+2k\sum_{n=1}^{\infty}\sigma_1(n)q^n.
\end{equation}
 Let us note that the constant term of $f_{\theta}$ at cusp $t$ is $\nu_t^{(N)}(f)-k/12$.
Suppose that $m$ is prime to $N$.
Then, Lemma \ref{basis} implies that
\begin{equation} \label{b_m}
b(m) = -24\left( \sum_{1 \leq j \leq D(N)-1 }\frac{\det(A_{f,j})}{\det(A_N)}\right) \sigma_1(m).
\end{equation}
Here, $A_{f,j}$ is a matrix obtained from $A_N$ by replacing the $j$th column of $A_N$ with a column matrix whose $i$th component is $\frac{\nu_{1/d_i}^{(N)}(f)}{\alpha_{1/d_i}}-\frac{k}{12}$.
Let us note that if $\gcd(m,N)=1$, then $J_{N,m}=J_{N,1}|T_m$.
Therefore, by (\ref{nonconstant}), (\ref{f_theta_expansion}), and (\ref{b_m}), we have
\begin{eqnarray*}
-J_{N,1}(T_m . D_f) &=& -J_{N,m}(D_f) =  -\sum_{\tau\in \mathcal{F}_N} \nu_\tau^{(N)}(f)J_{N,m}(\tau) = a_{\theta}(m) - b(m)\\
 &=& \sum_{d\mid m} dc(d) +24 \left( \sum_{1 \leq j \leq D(N)-1 }\frac{\det(A_{f,j})}{\det(A_N)} +  \frac{k}{12}\right)\sigma_1(m).
\end{eqnarray*}
\end{proof}
%%%

To prove Theorem \ref{main2}, we follow the argument of the proof of \cite[Proposition~3]{D}.
%In some steps of the proof, we need a more careful calculation to obtain an explicit upper bound for the error term.
We fix $\epsilon>0$.
Let $\psi_\epsilon : \RR_{>0} \to \RR$ be a $C^\infty$ function with $0\leq \psi_\epsilon(y)\leq 1$ for all $y\in \RR_{>0}$ and
\begin{equation*}
\psi_\epsilon(y) = \begin{cases}0, & \text{ if } y\leq 1, \\
1, & \text{ if } y> 1+\epsilon. \end{cases}
\end{equation*}
For a positive integer $n$,
consider the Poincar\'e series defined by
\begin{equation}\label{e:Pm}
P_{n,\epsilon}(z) := \sum_{\gamma\in \Gamma_0(N)_\infty \backslash \Gamma_0(N)} \psi_\epsilon(\mathrm{Im}(\gamma z)) e(-n(\gamma z)).
\end{equation}
From this, we obtain the following proposition.
%Actually, only one term in the series $P_m$ survives.
%%%

%\begin{lem}\label{lem:Pm}
%Let $z=x+iy\in \HH$.
%If $y>1$, we have
%\begin{equation}
%P_m(z) = \psi(y)e(-mz).
%\end{equation}
%If $y\leq1$ and there exist $c, d\in \ZZ$, $c\geq 1$, $\gcd(c, d)=1$, such that
%$\left(x+\frac{d}{c}\right)^2+\left(y-\frac{1}{2c^2}\right)^2 < \left(\frac{1}{2c^2}\right)^2$,
%then
%\begin{equation}
%P_m(z) = \psi\left(\frac{y}{|cz+d|^2}\right) e\left(-m \frac{az+b}{cz+d}\right),
%\end{equation}
%for $\sm a & b\\ c& d\esm \in \SL_2(\ZZ)$.
%Otherwise $P_m(z)=0$.
%\end{lem}
%%%
%\begin{proof}
%Note that for $\gamma=\sm a& b\\ c& d\esm\in \Gamma$,
%$\mathrm{Im}(\gamma z) = \mathrm{Im}(z)$ if $c=0$,
%and
%\[
%\mathrm{Im}(\gamma z) = \frac{\mathrm{Im}(z)}{|cz+d|^2} \leq \frac{1}{|c|^2\mathrm{Im}(z)}
%\]
%for $c\neq0$.
%Set $z=x+iy$.
%If $y>1$, then $\mathrm{Im}(\gamma z) \leq \frac{1}{|c|^2 y} < 1$ for $c\neq 0$.
%So, we have
%\begin{equation*}
%P_m(z) = \psi(y) e(-m z) \quad \text{ for } y>1.
%\end{equation*}
%
%
%For $y\leq 1$, if there exists $c, d\in \ZZ$, $c\geq1$ with $\gcd(c, d)=1$ such that
%$\left(x+\frac{d}{c}\right)^2+ \left(y-\frac{1}{2c^2}\right)^2 < \left(\frac{1}{2c^2}\right)^2$,
%then $\frac{y}{|cz+d|^2} >1$, so
%\begin{equation}
%P_m(z) = P_m\left(\sm a & b\\ c& d\esm z\right)= \psi\left(\frac{y}{|cz+d|^2}\right) e\left(-m\frac{az+b}{cz+d}\right),
%\end{equation}
%for $\sm a & b\\ c& d\esm \in \SL_2(\ZZ)$.
%Otherwise, $P_m(z)=0$.
%\end{proof}

\begin{prop} \label{dis_J}
Let $\theta$ be given as in Section \ref{s:heckepts}.
Fix  $n\in \ZZ_{\geq 1}$, $\epsilon>0$, and $z_0\in\HH$.
For any positive integer $m$ prime to $N$ and any $\epsilon'>0$, we have
\begin{multline*}
\left|\frac{1}{\sigma_1(m)} \left\{ \sum_{\substack{ad=m, \\ b\pmod{d}}}  J_{N,n}\left(\frac{az_0+b}{d}\right) - \sum_{\substack{ad=m, \\ b\pmod{d}}} P_{n,\epsilon}\left(\frac{az_0+b}{d}\right) \right\}
- \lim_{\epsilon''\to0}
\int_{\mathcal{F}_N(\epsilon'')} J_{N,n}(z) \; d\mu(z) \right|\\
\leq  C_{z_0, \epsilon'} m^{-\frac12+\theta+\epsilon'} \max\{\| F_{n,\epsilon}\|_2, \|\Delta^2 F_{n, \epsilon}\|_2\},
\end{multline*}
%as $n\to \infty$.
%Here, $\theta$ is the constant towards the Ramanujan conjecture.
where
$F_{n,\epsilon} := J_{N,n}-P_{n,\epsilon}$
and  $C_{z_0, \epsilon'}$ is the constant given as in Corollary \ref{pointwise_convergence}.
\end{prop}
%%%

\begin{proof}
%Fix $\epsilon>0$.
For a positive integer $n$, let $P_{n,\epsilon}$ be the Poincar\'e series as in \eqref{e:Pm}.
From \cite[Theorem 2.1]{C2}, it follows that  $F_{n,\epsilon} \in L^2(\Gamma_0(N) \backslash \mathbb{H})$ for a fixed $n \geq 1$. %the function
%\begin{equation*}
%F_{m,\epsilon} := J_{N,m}-P_{m,\epsilon}
%\end{equation*}
%where $d\mu(z) := \frac{3}{\pi [\mathrm{SL}_2(\mathbb{Z}):\Gamma_0(N)]}  \cdot \frac{dxdy}{y^2}$.

%
Recall that for $\phi\in L^2(\Gamma_0(N)\backslash \mathbb{H})$ and $m\geq 1$ with $\gcd(m,N)=1$, the normalized Hecke operator $T_m$ can be represented by
\begin{equation*}
T_m \phi(z) =  \sum_{\gamma\in T(m)} \phi\left(\gamma z\right).
\end{equation*}
%Here, $\sigma_1(n) = \sum_{d\mid n} d$, which is equal to $\#T(n)$.
%
%By Theorem \ref{L2_convergence}, we get
%\begin{equation*}
%\left\| \frac{1}{\sigma_1(n)}T_n F_m - \int_{\mathcal{F}_N} F_m(z) \; d\mu(z)\right\|_2
%\ll n^{-\frac{1}{2}+\epsilon' +\theta} \|F_m \|_2,
%\end{equation*}
%for any $\epsilon'>0$, where $\theta$ is the constant towards the Ramanujan conjecture.
%Note that we can take $\theta$ such that $\theta<\frac12$.
%
By Corollary \ref{pointwise_convergence}, we find that for any $\epsilon'>0$ and $m$
%this $L^2$-convergence implies the same rate for pointwise convergence: for any $z_0\in \HH$,
\begin{equation} \label{F_m1}
\left| \frac{1}{\sigma_1(m)}(T_m F_{n,\epsilon}) (z_0) - \int_{\mathcal{F}_N} F_{n,\epsilon}(z) \; d\mu(z)\right|
\leq C_{z_0, \epsilon'} m^{-\frac{1}{2}+\epsilon'+\theta} \max\{\|F_{n,\epsilon} \|_2, \|\Delta^2 F_{n, \epsilon}\|_2\}.
\end{equation}
%as $n\to\infty$,
%where $\theta$ is the constant towards the Ramanujan conjecture.
%

For $z_0\in \HH$, we have
\begin{equation} \label{F_m2}
\frac{1}{\sigma_1(m)} T_m F_{n,\epsilon}(z_0)
= \frac{1}{\sigma_1(m)} \left\{ \sum_{\substack{ad=m, \\ b\pmod{d}}}  J_{N,n}\left(\frac{az_0+b}{d}\right) - \sum_{\substack{ad=m, \\ b\pmod{d}}} P_{n,\epsilon}\left(\frac{az_0+b}{d}\right) \right\}.
\end{equation}
Note that
\begin{equation*}
\lim_{\epsilon''\to0} \int_{\mathcal{F}_N(\epsilon'')} P_{n,\epsilon}(z) \; \frac{dx\; dy}{y^2}
= \int_{\Gamma_0(N)_\infty \backslash \HH} \psi_\epsilon(y) e^{-2\pi inz} \frac{dxdy}{y^2}
= \int_0^\infty \psi_\epsilon(y) e^{2\pi ny} \; \frac{dy}{y^2} \cdot \int_{0}^1 e^{2\pi inx} \; dx
=0
\end{equation*}
for every positive integer $n$.
So,  we get
\begin{equation} \label{F_m3}
\int_{\mathcal{F}_N} F_{n,\epsilon}(z) \; d\mu(z) = \lim_{\epsilon''\to0}\int_{\mathcal{F}_N(\epsilon'')} J_{N,n}(z) \; d\mu(z).
\end{equation}
If we combine (\ref{F_m1}), (\ref{F_m2}), and (\ref{F_m3}), then we get the desired result.
\end{proof}

We define
\[
Q_{1,\epsilon}(z) := \psi_{\epsilon}(\mathrm{Im}(\tilde{z}))e(-\tilde{z})
\]
for $\epsilon>0$. Then, we obtain the following proposition.

\begin{prop}  \label{first_part}
For any $m$ with $\gcd(m, N)=1$,
we obtain
\begin{multline} \label{temp1}
\left| \frac{1} {\sigma_1(m)} \left( J_{N,1}(T_m . D_f) - Q_{1,\epsilon}(T_m . D_f) \right) -  h_f \lim_{\epsilon''\to0}
\int_{\mathcal{F}_N(\epsilon'')} J_{N,1}(z)d\mu(z)\right|
\\ \leq
H_f C(f, \epsilon') m^{-\frac12 + \theta+\epsilon'}\max\{\| F_{1,\epsilon} \|_2, \|\Delta^2 F_{1, \epsilon}\|_2\},
\end{multline}
where  $h_f$ denotes  the sum of the orders of zero or pole of $f$ at $Q_{\tau}$ on $Y_0(N)$.
\end{prop}

\begin{proof}
Let $\epsilon>0$ be fixed.
Note that
\[
J_{N,1}(T_m . D_f) = \sum_{\tau\in \mathcal{F}_N} \nu_{\tau}^{(N)}(f) \sum_{ad=m \atop b \pmod{d}} J_{N,1}\left(\frac{a\tau+b}{d}\right)
\]
and
\begin{eqnarray*}
P_{1,\epsilon}(T_m . D_f) &=& \sum_{\tau\in \mathcal{F}_N} \nu_{\tau}^{(N)}(f) \sum_{ad=m, \atop b \pmod{d}} P_{1,\epsilon}\left(\frac{a\tau+b}{d}\right).
\end{eqnarray*}
Therefore, by Proposition~\ref{dis_J},  for any $m$ with $\gcd(m,N)=1$, we have
\begin{multline} \label{limit_P}
\left| \frac{1} {\sigma_1(m)} \left( J_{N,1}(T_m . D_f) - P_{1,\epsilon}(T_m . D_f) \right) -h_f \lim_{\epsilon''\to0} \int_{\mathcal{F}_N(\epsilon'')} J_{N,1}(z)d\mu(z)\right|
\\ \leq
H_f C(f, \epsilon') m^{-\frac12 + \theta+\epsilon'}\max\{\| F_{1,\epsilon} \|_2, \|\Delta^2 F_{1, \epsilon}\|_2\}, \end{multline}
%as $m\to\infty$.
for any  $\epsilon'>0$,
where $H_f := \sum_{\tau\in \mathcal{F}_N}  |\nu_\tau^{(N)}(f)|$
and   $C(f,\epsilon') := \mathrm{max} \{ C_{\tau, \epsilon'}\ |\ \tau\in \mathcal{F}_N,\ \nu_{\tau}^{(N)}(f)\neq0\}$.
%Here, $F_{1, \epsilon} = J_{N, 1}-P_{1, \epsilon}$.

%
%\red (I cannot understand this part)
%If $z\in \mathcal{F}_N$, then $\mathrm{Im}(\gamma z) \leq 1$ for any $\gamma \not\in \Gamma_0(N)$.
%From this, we have
%\[
%P_{1,\epsilon}(z) = P_{1,\epsilon}(\tilde{z}) = \psi_{\epsilon}(\mathrm{Im}(\tilde{z}))e(-\tilde{z})
%\]
%since $\tilde{z}\in\mathcal{F}_N$ is equivalent to $z$ under the action of $\Gamma_0(N)$ and $P_{1,\epsilon}$ is modular of weight $0$ on $\Gamma_0(N)$.
%\endred
%

Recall that $\tilde{z}$ is a unique complex number in $\mathcal{F}_N$ which is equivalent to $z$ under the action of $\Gamma_0(N)$.
%For any $z\in \HH$, there exists a uniquely determined $\tilde{z}\in \mathcal{F}_N$, which is equivalent to $z$ under the action of $\Gamma_0(N)$.
If $\im(\tilde{z}) >1$, then for any $\gamma\in \Gamma_0(N)$, $\im(\gamma \tilde{z}) \leq 1$ unless $\gamma\in \Gamma_0(N)_\infty$.

Suppose that $\im(\tilde{z}) \leq 1$ and that there exists $\gamma\in \Gamma_0(N)$ such that $\im(\gamma \tilde{z})>1$.
Then, there exists $\ell\in \ZZ$ such that $-\frac{1}{2} < \mathrm{Re}(\gamma\tilde{z}) +\ell \leq \frac{1}{2}$, and so
\[ \gamma \tilde{z} + \ell = \mat 1 & \ell \\ 0 & 1\emat \gamma \tilde{z} \in \mathcal{F}_N. \]
Since $\sm 1 & \ell \\ 0 & 1\esm \gamma\in \Gamma_0(N)$, we have $\gamma \tilde{z}+\ell=\tilde{z}$,
so $\im(\gamma \tilde{z}) = \im(\tilde{z}) \leq 1$,
which is a contradiction.
Therefore, if $\im(\tilde{z})\leq 1$, then for any $\gamma\in \Gamma_0(N)$, we get $\im(\gamma \tilde{z}) \leq 1$.

%
%If $\Im(\tilde{z}) \leq 1$ then for any $\gamma\in \Gamma_0(N)$, $\Im(\gamma \tilde{z}) \leq 1$.

Thus, we have
\[
P_{1, \epsilon}(z) = P_{1, \epsilon}(\tilde{z})
= \sum_{\gamma\in \Gamma_0(N)_\infty \bsl \Gamma_0(N)} \psi_{\epsilon}(\im(\gamma \tilde{z})) e(-\gamma \tilde{z})
= Q_{1, \epsilon}(z).
\]
Therefore, from (\ref{limit_P}),
we obtain the desired result.
\end{proof}

From Proposition \ref{dis_J} and Proposition \ref{first_part}, we obtain the following theorem.
This gives the distribution of values of $J_{N,1}$ on Hecke orbits.

\begin{thm} \label{J_version}
%Let $m$ be a positive integer prime to $N$.  Then,
%For any $\epsilon'>0$, we have
We have
\begin{equation*}
\lim_{m \rightarrow \infty}
\frac{1}{\sigma_1(m)} \left ( J_{N,1}(T_m . D_f) - e\left(\left(T_m . D_f\right)_{\; >1}\right) \right)\\
=\frac{3h_f}{\pi [\mathrm{SL}_2(\mathbb{Z}):\Gamma_0(N)]} \lim_{\epsilon''\to0} \int_{\mathcal{F}_{N}(\epsilon'')}  J_{N,1}(z)\frac{dxdy}{y^2}.
%+ O\left(m^{-\frac12+\theta+\epsilon'}\right)
\end{equation*}
%as $m\to\infty$,
%where  $h_f$ denotes  the sum of the orders of zero or pole of $f$ at $Q_{\tau}$ on $Y_0(N)$.
\end{thm}
%%%
\begin{proof}
%We will show that for given $\tilde{\epsilon}>0$ there is $\tilde{m}$ such that
%\begin{equation} \label{full_part}
%\left| \frac{1}{\sigma_1(m)}\left ( J_{N,1}(T_m\cdot D_f) - e\left(\left(T_m\cdot D_f\right)_{\; >1}\right) \right)- \lim_{\epsilon''\to0} \int_{\mathcal{F}_{N}(\epsilon'')}  J_{N,1}(z) d\mu(z)\right| < \tilde{\epsilon}
%\end{equation}
%for all $m\geq \tilde{m}$ with $(m,N)=1$.
%Suppose that $\tilde{\epsilon}>0$ is given.
Let $\epsilon>0$ be fixed.
For any positive integer $m$ which is  prime to $N$, we have
\begin{equation} \label{separated_equation}
\begin{aligned}
&\left| \frac{1}{\sigma_1(m)}\left ( J_{N,1}(T_m . D_f) - e\left(\left(T_m . D_f\right)_{\; >1}\right) \right)- h_f\lim_{\epsilon''\to0} \int_{\mathcal{F}_{N}(\epsilon'')}  J_{N,1}(z) d\mu(z)\right|\\
&\leq \left| \frac{1}{\sigma_1(m)}\left ( J_{N,1}(T_m . D_f) -   Q_{1,\epsilon}(T_m . D_f)       \right)- h_f\lim_{\epsilon''\to0} \int_{\mathcal{F}_{N}(\epsilon'')}  J_{N,1}(z) d\mu(z)\right|\\
&+ \frac{1}{\sigma_1(m)}|Q_{1,\epsilon}(T_m . D_f) -  e\left(\left(T_m . D_f\right)_{\; >1}\right)|.
\end{aligned}
\end{equation}
%where $Q_{1,\epsilon}(z) := \psi_\epsilon(\mathrm{Im}(\tilde{z})) e(-\tilde{z})$.
%Recall that $\tilde{z}$ is a complex number in $\mathcal{F}_N$, which is equivalent to $z$ under the action of $\Gamma_0(N)$.

%
Note that
\begin{multline}\label{change_summation}
\left|Q_{1, \epsilon}(T_m. D_f) - e(-(T_m. D_f))_{>1}\right|
\\  \leq
\sum_{\substack{\tau\in \mathcal{F}_N \atop \nu_\tau^{(N)}(f)\neq 0}}  \left|\nu_{\tau}^{(N)}(f)\right| \sum_{\gamma\in T(m)}
\begin{cases}
\left|\psi_{\epsilon}(\im(\widetilde{\gamma\tau})-1\right| |e(-\widetilde{\gamma\tau})|, & \text{ if } 1 < \im(\widetilde{\gamma \tau}) \leq 1+\epsilon, \\
0, & \text{ otherwise.}
\end{cases}
\end{multline}
Now, we follow the proof of \cite[Proposition~3]{D}.
Fix $0< \epsilon < \frac{1}{4}$ and consider the incomplete Eisenstein series
%
%Note that
%\begin{equation} \label{change_summation}
%\begin{aligned}
%Q_{1,\epsilon}(T_m . D_f) -  e\left(-\left(T_m . D_f\right)_{\; >1}\right)& =
% \sum_{\tau\in A_{f,m,\epsilon}} Q_{1,\epsilon}(\tau) - \sum_{\tau\in A_{f,m,\epsilon}} e(-\tilde{\tau}),
%\end{aligned}
%\end{equation}
% where
%\[
%A_{f,m,\epsilon} := \{ \gamma\tau\ |\ \tau\in X_0(N),\ \nu_\tau^{(N)}(f)\neq 0,\ \gamma\in T(m),\ \text{and}\
%1 < \mathrm{Im}(\widetilde{\gamma\tau}) \leq 1+\epsilon\}.
%\]
%Suppose that $0<\epsilon<\frac 14$ is fixed.
%Consider the incomplete Eisenstein series
\[
g_\epsilon(z) := \sum_{\gamma\in \Gamma_0(N)_\infty \backslash \Gamma_0(N)} \phi_\epsilon(\mathrm{Im}\gamma z),
\]
where $\phi_\epsilon:\RR_{>0}\to\RR$ is a  smooth  function supported in $(1-\epsilon, 1+2\epsilon)$ with $0\leq \phi_\epsilon(y)\leq 1$ for all $y\in \RR_{>0}$ and $\phi_\epsilon(y) = 1$ for $1\leq y\leq 1+\epsilon$.
By Corollary~\ref{pointwise_convergence}, we see that for any $\epsilon'>0$, $z_0\in \HH$, and $m$ with $\gcd(m,N)=1$,
\begin{equation} \label{g_epsilon}
\left| \frac{1}{\sigma_1(m)} T_m g_\epsilon(z_0) - \int_{\mathcal{F}_N} g_\epsilon(z)d\mu(z) \right|
\leq  C_{z_0, \epsilon'} m^{-\frac{1}{2}+\epsilon'+\theta} \max\{\|g_\epsilon \|_2, \|\Delta^2 g_{\epsilon}\|_2\}.
\end{equation}
Then, there exists a constant $D(f, \epsilon')$ such that
\begin{multline}
\frac{1}{\sigma_1(m)} \left|Q_{1, \epsilon}(T_m.D_f)- e(-(T_m.D_f)_{>1})\right|
\leq \frac{e^{2\pi(1+\epsilon)}}{\sigma_1(m)} \sum_{\substack{\tau\in \mathcal{F}_N \atop \nu_\tau^{(N)}(f)\neq 0}} \left|\nu_{\tau}^{(N)}(f)\right| (T_m g_\epsilon)(\tau)
\\ \leq
C_f e^{2\pi(1+\epsilon)} \left(\int_{\mathcal{F}_N} g_{\epsilon}(z) \; d\mu(z) + D(f,\epsilon') m^{-\frac{1}{2}+\theta+\epsilon'}
\max\{\|g_\epsilon\|_2, \|\Delta^2 g_{\epsilon}\|_2\}\right),
\end{multline}
where $C_f := \#\left\{\tau\in \mathcal{F}_N \ \bigg|\ \nu_\tau^{(N)}(f)\neq0\right\} \times \max\left\{\left|\nu_\tau^{(N)}(f)\right|\ \bigg|\ \tau\in \mathcal{F}_N\right\}$.

Therefore, from Proposition \ref{first_part} and (\ref{separated_equation}), we have
\begin{equation} \label{e:concl}
\begin{aligned}
&\left| \frac{1}{\sigma_1(m)}\left ( J_{N,1}(T_m . D_f) - e\left(\left(T_m . D_f\right)_{\; >1}\right) \right)- h_f\lim_{\epsilon''\to0} \int_{\mathcal{F}_{N}(\epsilon'')}  J_{N,1}(z) d\mu(z)\right|
\\ &\leq
\bigg(H_f C(f, \epsilon')\max\{\|F_{1, \epsilon}\|_2, \|\Delta^2 F_{1, \epsilon}\|_2\}
+C_f e^{2\pi(1+\epsilon)}D(f, \epsilon')\max\{\|g_{\epsilon}\|_2, \|\Delta^2 g_{\epsilon}\|_2\} \bigg)
m^{-\frac{1}{2}+\theta+\epsilon'}\\
&\qquad +C_f e^{2\pi(1+\epsilon)} \int_{\mathcal{F}_N} g_\epsilon(z) \; d\mu(z).
\end{aligned}
\end{equation}
For a fixed $\epsilon$, taking $m\to \infty$, we get
\begin{multline}
\lim_{m\to \infty}
\frac{1}{\sigma_1(m)}\left ( J_{N,1}(T_m . D_f) - e\left(\left(T_m . D_f\right)_{\; >1}\right) \right)\\
=
h_f\lim_{\epsilon''\to0} \int_{\mathcal{F}_{N}(\epsilon'')}  J_{N,1}(z) d\mu(z)
+ C_f e^{2\pi(1+\epsilon)} \int_{\mathcal{F}_N} g_\epsilon(z) \; d\mu(z).
\end{multline}
Note that \eqref{e:concl} holds for any fixed $0< \epsilon < \frac{1}{4}$.
Since
\[ \int_{\mathcal{F}_N} g_{\epsilon}(z) d\mu(z)
= \frac{3}{\pi[\SL_2(\mathbb{Z}) : \Gamma_0(N)]}
\int_0^\infty  \phi_\epsilon(y) \; \frac{dy}{y^2} \to 0, \]
as $\epsilon\to 0$, we get
\[
\lim_{m\to \infty}
\frac{1}{\sigma_1(m)}\left ( J_{N,1}(T_m . D_f) - e\left(\left(T_m . D_f\right)_{\; >1}\right) \right)
=
h_f\lim_{\epsilon''\to0} \int_{\mathcal{F}_{N}(\epsilon'')}  J_{N,1}(z) d\mu(z).
\]

\end{proof}

Finally, Theorem \ref{main2} comes from Theorem \ref{J_version} and (\ref{J_exponents}).

%%%%%%%%%%%%%%%%%%%%%%%%%%%%%%%%%%%%%%%%%%%%%%%

\end{document}